\documentclass{amsart}
\usepackage{enumerate}
\usepackage{colortbl}
\newtheorem{theorem}{Theorem}[section]
\newtheorem{lemma}[theorem]{Lemma}
\theoremstyle{definition}
\newtheorem{definition}[theorem]{Definition}
\newtheorem{example}[theorem]{Example}

\newtheorem{corollary}[theorem]{Corollary}

\theoremstyle{remark}
\newtheorem{remark}[theorem]{Remark}

\numberwithin{equation}{section}

\begin{document}

\def\C{\mathbb C}
\def\R{\mathbb R}
\def\X{\mathbb X}
\def\Z{\mathbb Z}
\def\Y{\mathbb Y}
\def\Z{\mathbb Z}
\def\N{\mathbb N}
\def\cal{\mathcal}

\def\cal{\mathcal}
\def\b{\mathcal B}
\def\c{\mathcal C}
\def\cc{\mathbb C}
\def\x{\mathbb X}
\def\r{\mathbb R}
\def\uu{(U(t,s))_{t\ge s}}
\def\vv{(V(t,s))_{t\ge s}}
\def\xx{(X(t,s))_{t\ge s}}
\def\yy{(Y(t,s))_{t\ge s}}
\def\zz{(Z(t,s))_{t\ge s}}
\def\ss{(S(t))_{t\ge 0}}
\def\tt{(T(t,s))_{t\ge s}}
\def\rr{(R(t))_{t\ge 0}}

\title[Massera Theorem]{
On asymptotic periodic solutions of fractional differential equations and applications}
\author[V.T Luong]{Vu Trong Luong}
\address{VNU University of Education, Vietnam National University at Hanoi, 144 Xuan Thuy, Cau Giay, Hanoi, Vietnam}
\email{vutrongluong@gmail.com}

\author[N.D. Huy]{Nguyen Duc Huy}
\address{VNU University of Education, Vietnam National University at Hanoi, 144 Xuan Thuy, Cau Giay, Hanoi, Vietnam}
\email{huynd@vnu.edu.vn}

\author[N.V. Minh]{Nguyen Van Minh}
\address{Department of Mathematics and Statistics, University of Arkansas at Little Rock, 2801 S University Ave, Little Rock, AR 72204, USA}
\email{mvnguyen1@ualr.edu}

\author[N.N. Vien]{Nguyen Ngoc Vien}
\address{Faculty of Foundations, Hai Duong University, Hai Duong City, Vietnam}
\email{uhdviennguyen.edu@gmail.com}

\date{\today}

\thanks{A part of this article was written while the first author was	visiting Vietnam Institute for Advanced Study in Mathematics (VIASM). He	would like to thank the Institute for its hospitality.}
\thanks{The authors are grateful to the anonymous referee for his carefully reading the manuscript and suggestions to improve its presentation}
\begin{abstract} In this paper we study the asymptotic behavior of solutions of fractional differential equations of the form
$
D^{\alpha}_Cu(t)=Au(t)+f(t), u(0)=x, 0<\alpha\le1, ( *)
$
where $D^{\alpha}_Cu(t)$ is the derivative of the function $u$ in the Caputo's sense, $A$ is a linear operator in a Banach space $\X$ that may be unbounded and $f$ satisfies the property that $\lim_{t\to \infty} (f(t+1)-f(t))=0$ which we will call asymptotic $1$-periodicity. By using the spectral theory of functions on the half line we derive analogs of Katznelson-Tzafriri and Massera Theorems. Namely, we give sufficient conditions in terms of spectral properties of the operator $A$ for all asymptotic mild solutions of Eq. (*) to be asymptotic $1$-periodic, or there exists an asymptotic mild solution that is asymptotic $1$-periodic.
\end{abstract}
\maketitle

\section{Introduction} \label{section 1}
We are interested in the asymptotic behavior of solutions to linear fractional differential equations of the form
\begin{equation}\label{eq:0}
D^{\alpha}_Cu(t)=Au(t)+f(t), u(0)=x, 0<\alpha\le1,
\end{equation}
where $D^{\alpha}_Cu(t)$ is the derivative of the function $u$ in the Caputo's sense, $A$ is a linear operator in a Banach space $\X$ that may be unbounded and $f$ satisfies the property that $\lim_{t\to \infty} (f(t+1)-f(t))=0$ which we will call "asymptotic $1$-periodicity". 

\medskip
In recent decades fractional differential equations are of growing interests as this kind of equations allows us to model more complex processes. We refer the reader to the monographs \cite{hil,kilsritru} for an account of applications in Physics and Engineering. For general results and concepts in abstract spaces the reader is referred to \cite{baemeenan,baz,clegrilon}. 

\medskip
Asymptotic behavior of solutions of differential equations is a central topic in the qualitative theory of differential equations and dynamical systems that is rich of results and ideas. The very first behavior of solutions that are of interest to mathematicians because of its applicability is the periodicity of solutions. There are many classic results in this direction that are models for subsequent researches. Among many such works we refer the reader to very well known results by Massera \cite{mas} and Katznelson-Tsafriri \cite{kattza} due to their very different techniques of study. These results are recently extended to  wider classes of equations and solutions in \cite{luoloiminmat} and \cite{PAMS1}.
It is natural to ask if the concepts and results in these recent works could be extended to a larger class of Eq.(\ref{eq:0}). In this paper we will take an attempt to address this question.

\medskip
More specifically, we look for conditions under which 
some (or all) solutions $x(\cdot )$ are asymptotic $1$-periodic. We will prove an analog of the Katznelson-Tzafiri Theorem and Massera Theorem for this kind of solutions for Eq.(\ref{eq:0}). Our method is to use the spectral theory of functions and the technique of decomposition rather than the period map as it is often used in the area since the equation is actually not periodic and its homogeneous equation does not generate a dynamical system even in the autonomous case. In the recent paper \cite{luoloiminmat} similar problems in case $\alpha =1$ were studied for evolution equations with periodic coefficients based on a spectral theory of functions. Namely, in \cite{luoloiminmat} the spectrum of a function $f\in BUC(\R^+,\X)$ (the space of all uniformly continuous and bounded functions on $[0,\infty )$) is defined to be the subset $\sigma (f)$ of the unit circle $\Gamma$ in the complex plane consisting of all points $\xi_0$ such that the complex function $R(\lambda ,\bar S)\bar f$ has no analytic extension to any neighborhood of $\xi_0$. Here $S$ is the translation by $1$ and $\bar S$ is the induced operator by $S$ in the quotient space $\Y:=BUC(\R^+,\X)/C_0(\R^+,\X)$ (the space of continuous functions approaching zero at infinity and $R(\lambda ,\bar S)$ stands for the resolvent of $\bar S$. In \cite{luoloiminmat}, the asymptotic $1$-periodicity of a function $f\in BUC(\R^+,\X)$ is characterized by the property $\sigma (f)\subset \{ 1\}$. Note that for fractional equations (\ref{eq:0}), using the transform $R(\lambda ,\bar S)\bar f$ to capture the spectrum of an asymptotic mild solution of (\ref{eq:0}) does not work. As shown in \cite{luomin}, the spectrum $sp(f)$ determined by the transform $R(\lambda ,\bar {\cal D})\bar f$, where ${\cal D}$ is the differential operator $d/dt$ and $\bar {\cal D}$ stands for the induced operator by ${\cal D}$ in $\Y$, can be used to study the asymptotic behavior of Eq.(\ref{eq:0}). Our next step is to characterize the asymptotic $1$-periodicity in terms of $sp(f)$. This will be done by proving the identity $\sigma (f) =\overline{e^{i \cdot sp(f)}}$ that is based on the Weak Spectral Mapping Theorem for $C_0$-groups of isometries. This characterization of the asymptotic $1$-periodicity allows us to study the asymptotic $1$-periodicity of all asymptotic mild solutions existence of Eq.(\ref{eq:0}), an analog of the Katznelson-Tzifriri for fractional equations. We also prove an analog of Massera Theorem for this class of equations based on the technique of decomposition. Our main results are stated in Theorems \ref{the main}, \ref{the main2} and \ref{the main3}.

\section{Preliminaries and Notations}
\subsection{Notations}
Throughout this paper we will denote by $\R,\R^+, \C$ the real line $(-\infty,\infty)$, half line $[0,+\infty )$ and the complex plane. For $J$ being either $\R$ or $\R^+$, the notation $BUC(J,\X)$ stands for the function space of all bounded and uniformly continuous functions taking values in a (complex) Banach space $\X$ with sup-norm. Throughout the paper we will use the following notations
\begin{align*}
C_{0}(\R^+,\X) &:=\{ f\in BUC   (\R^+,\X): \ \lim_{t\to \infty}\| f(t)\|   =0\} \\
g_\alpha(t) &:=\frac{t^{\alpha-1}}{\Gamma(\alpha)}, t>0,\alpha>0.
 \end{align*}
For a complex number $z$, $\Re z$ denotes its real part and $\Im z$ stands for its imaginary part. In this paper the single valued power function $\lambda^\alpha$ of the complex variable $\lambda$ is uniquely defined as
$\lambda^\alpha =| \lambda | ^\alpha e^{ i \alpha\cdot arg (\lambda)}$, with $-\pi < arg (\lambda) < \pi$.

\subsection{Fractional differentiation in Caputo's sense}
Let $\alpha >0, t\ge a,$ and $a$ is a fixed number. Then, the fractional operator
\begin{align}
J^{\alpha}_a u(t):=(g_\alpha \ast u)(t)=\displaystyle\int_{a}^{t}g_\alpha(t-\tau)u(\tau)d\tau
\end{align}
is called fractional Riemann-Liouville integral of degree $\alpha$. The function
\begin{equation*}
D^{\alpha}_{C}u(t):=
\begin{cases}
J^{n-\alpha}u^{(n)}(t)=\dfrac{1}{\Gamma(n-\alpha)}\displaystyle\int_{a}^t\dfrac{u^{(n)}(\tau)}{(t-\tau)^{\alpha+1-n}}d\tau, &{n-1}<{\alpha}<n\in\mathbb{N},\\
u^{(n)}(t),&\alpha=n\in\mathbb{N},
\end{cases}
\end{equation*}
 is called the fractional derivative in Caputo's sense of degree $\alpha$. By this notation we have for $0<\alpha \le 1$
 $$J^\alpha _a D^{\alpha}_C u(t)=u(t)-u(a).$$

\subsection{Cauchy Problem}
For a fixed $0<\alpha \le 1,$ consider the Cauchy problem
\begin{equation}\label{fde}
D^{\alpha}_Cu(t)=Au(t), u(0)=x, 
\end{equation}
where $A$ is generally an unbounded linear operator.
The well-posedness of \eqref{fde}  is equivalent to that of the problem
\begin{equation}\label{eq:4}
u(t)=x+\int_0^tg_\alpha(t-s)Au(s)ds.
\end{equation}
The reader is referred to the monograph \cite{pru} for an extensive study of the well-posedness of this kind of equations when $A$ is generally an unbounded operator. Recent extensions for more general equations can be found in \cite{keylizwar} and their references for more details.

\medskip
Let us consider inhomogeneous linear equations of the form
\begin{equation}\label{eq:2a}
D^{\alpha}_Cu(t)=Au(t)+f(t), t\geq 0,
\end{equation}
where $f\in BUC(\R^+,\X)$ is given.

\begin{definition}
A mild solution $u$ of Eq.(\ref{eq:2a}) on $\R^+$ is a continuous function on $\R^+$ such that, for each $t\in\R^+$, $J^\alpha u(t) \in D(A)$ and
$$u(t)=AJ^\alpha u(t)+J^\alpha f(t)+u(0).$$
\end{definition}

\subsection{A spectral theory of bounded Functions}
In this paper we will use the spectral theory of functions defined on the half line as it is outlined in \cite{luomin}. The spectrum of a bounded function on the half line is defined as below:

\medskip
Let $\cal D$ denote the differentiation operator $d/dt$ in $BUC  (\R^+,\X)$ with domain
$$
D(\cal D) =\{ f\in BUC (\R^+,\X): \exists f', \ f'\in BUC  (\R^+,\X)\} .
$$
 \begin{lemma}
 The translation semigroup $(S(t)_{t\ge 0})$ in $BUC  (\R^+,\X)$ is strongly continuous with  ${\cal D}$ as its infinitesimal generator. 
 \end{lemma}
 \begin{proof}
 This is a well known result (see e.g. \cite{engnag}).
 \end{proof}

\subsection{Operator $\tilde{\cal D}$} In this subsection we recall some construction made in \cite{luomin}.
Since $C_{0  }(\R^+,\X)$ is a closed subspace $BUC   (\R^+,\X)$, and is invariant under the translation semigroup $(S(t)_{t\ge 0})$.
In the space $BUC   (\R^+,\X)$ we introduce the following relation $R$:
\begin{equation}
 f \ R \ g \ \mbox{if and only if} \ \ f -g \in C_{0  }(\R^+,\X) .
\end{equation}
This is an equivalence relation and the quotient space $\Y:= BUC   (\R^+,\X)/ R$ is a Banach space. We will also denote the norm in this quotient space $\Y$ by $\| \cdot \| $ if it does not cause any confusion.

\medskip
The class containing $f\in BUC   (\R^+,\X)$ will be denoted by $\tilde{f}$. Define $\tilde{\cal D}$ in $ \Y=BUC   (\R^+,\X)/ R$ as follows:
\begin{eqnarray}
D( \tilde{\cal D}) &:=&\{ \tilde{f} \in \Y : \exists u\in \tilde{f}, u\in D(\cal D)\} .
\end{eqnarray}
If $f\in D( \tilde{\cal D}) $, we set
\begin{equation}
\tilde{D}\tilde{f} := \widetilde {\cal Du} 
\end{equation}
for some $u\in \tilde f$. The following lemma will show that this $\tilde{D}$ is well defined as an operator in $\Y$.
\begin{lemma}
With the above notations, $\tilde{\cal D}$ is a well defined single valued linear operator in $\Y$. \end{lemma}
\begin{proof}
For the proof see \cite{luomin}.
\end{proof}

For each given $f\in BUC  (\R^+,\X)$ consider the following complex function $\hat f(\lambda )$ in $\lambda$ defined as
\begin{equation}
\hat f (\lambda ) := (\lambda -\tilde{\cal D})^{-1} \tilde{f} .
\end{equation}

\begin{definition}
The set of all points $\xi_0 \in \R$ such that $\hat f (\lambda )$ has no analytic extension to any neighborhood of $i\xi_0$ is defined to be the spectrum of $f$, denoted by $sp  (f)$.
\end{definition}

\section{Main results}
\subsection{Spectral characterization of asymptotic $1$-periodicity}
\begin{definition}
Let $p$ be a given real number in $[0,2\pi)$. A function $g\in BUC(\R^+,\X)$ is said to be an asymptotic Bloch $1$-periodic function of type $p$ if
\begin{equation*}
\lim_{t\to\infty} (g(t+1)-e^{ip}g(t))=0.
\end{equation*}
If $p=0$, an asymptotic Bloch $1$-periodic function $g$ of type $p$ will be called an asymptotic $1$-periodic function. When $p=\pi$ we call the function asymptotic anti $1$-periodic.
\end{definition}
Below we will show that the spectrum of a function $f\in BUC  (\R^+,\X)$ can be determined as the spectrum of a linear operator. Let us define ${\cal M}_f$ as the closure of the linear space spanned by the set $\{ \tilde S(t)f:\ t\in \R \} \subset \Y$. Obviously $\bar S(t)$ leaves ${\cal M}_f$ invariant for every $t\in \R$. We consider the $C_0$-group of isometries $\{ \tilde S|_{{\cal M}_f}, t\in \R \}$. The generator of this group will be denoted by $\tilde D|_{{\cal M}_f}$.

\begin{lemma}
For each $f\in BUC  (\R^+,\X)$, we have
\begin{equation}
i \cdot sp (f) =\sigma (\tilde {\cal D}|_{{\cal M}_f}) .
\end{equation}
\end{lemma}
\begin{proof}
If $\lambda_0\in \rho (\tilde D|_{{\cal M}_f})$, then, for $\lambda\in U(\lambda_0)\backslash i\R$, where $U(\lambda_0)$ is a small neighborhood of $\lambda_0$
$$
\hat f (\lambda ) = (\lambda -\tilde{\cal D})^{-1} \tilde{f} = (\lambda -\tilde{\cal D}_{{\cal M}_f})^{-1} \tilde{f} .
$$
Therefore, $\hat f (\lambda ) $ has $ (\lambda -\tilde{\cal D}_{{\cal M}_f})^{-1} \tilde{f} $ as
an analytic extension to a neighborhood of $\lambda_0$. That means $\lambda_0\not\in i \cdot sp (f)$.

\medskip
Conversely, let $\lambda_0\in i\R \backslash i \cdot sp (f)$. Then, $\hat f (\lambda ) $ has an analytic extension to a connected neighborhood $U(\lambda_0)$ of $\lambda_0$. We will show that the equation
\begin{equation}\label{3.6}
\lambda_o y-\tilde {\cal D} y =w
\end{equation}
has a unique solution $y\in {\cal M}_f$ for each $w\in {\cal M}_f$. First, we show that Eq.(\ref{3.6}) has a solution. For $\lambda \in U(\lambda_0) \backslash i\R$ such that $\Re \lambda>0$, using the formula
$$
R(\lambda ,\tilde {\cal D})\tilde f =\int^\infty _0 e^{-\lambda t}\tilde S(t)\tilde fdt,
$$
we see that $\hat f(\lambda )=R(\lambda ,\tilde D)\tilde f  \in {\cal M}_f$. Subsequently, $\hat f(\lambda_0) \in  {\cal M}_f$. As the operator $\tilde D$ is a closed operator and $\hat f(\lambda)$ is analytic at $\lambda_0$, we see that
\begin{equation}
\lambda_0 \hat f(\lambda_0) -\tilde {\cal D}\hat f(\lambda _0) =w.
\end{equation}
Next, we will show that the solution of Eq.(\ref{3.6}) is unique for each given $w\in {\cal M}_f$. This is equivalent to show that the homogeneous equation has only trivial solution. Indeed, suppose $y_0\in {\cal M}_f$ is a solution of 
$$
\lambda_0 y_0-\tilde {\cal D} y_0 =0.
$$
Using the maximum modulus principle of holomorphic functions as in \cite[Proof of Theorem 2.2,p. 2074]{batneerab2} and then the Vitali's Theorem on convergence of sequences of holomorphic functions we can show that for each $w\in {\cal M}_f$, the function $\hat w(\lambda )$ has an analytic extension in the connected neighborhood $U(\lambda_0)$ of $\lambda _0$.
Next, for $\lambda \in U(\lambda_0)\backslash i\R$ by the identity
$$
R(\lambda ,\tilde {\cal D}) (\lambda -\tilde {\cal D})y_0 =y_0
$$
we have
$$
\lambda R(\lambda ,\tilde {\cal D})y_0 -y_0= R(\lambda ,\tilde {\cal D})\tilde {\cal D}y_0= R(\lambda ,\tilde {\cal D})\lambda_0 y_0 .
$$
Hence
\begin{equation}
\hat y_0 (\lambda) =R(\lambda ,\tilde {\cal D})y_0 =\frac{y_0}{\lambda -\lambda _0}.
\end{equation}
This function has an analytic extension to a neighborhood of $\lambda_0$ if and only if $y_0=0$. That is, $\lambda_0$ is in $\rho (\tilde {\cal D}|_{{\cal M}f})$.
This completes the proof of the lemma.
\end{proof}
Recall that in \cite{PAMS1} the "circular" spectrum $\sigma (f)$ of $f\in BUC  (\R^+,\X)$ is defined to be the set of $\lambda_0\in \Gamma$ such that the complex function $R(\lambda ,\tilde S(1))\tilde f$ has no analytic extension to any neighborhood of $\lambda_0$. In the same way as \cite[Lemma 2.7]{minPAMS2009} we can show that
\begin{lemma}
For $f\in BUC  (\R^+,\X)$, the following is valid
\begin{equation}
\sigma (f) =\sigma (\tilde S(1)|_{{\cal M}_f}) .
\end{equation}
\end{lemma}
\begin{corollary}\label{cor 1}
For $f\in BUC  (\R^+,\X)$ we have
\begin{equation}\label{spec}
\sigma (f) =\overline{e^{i \cdot {\rm sp}(f)}},
\end{equation}
where the overlining means the closure in the topology of the complex plane.
\end{corollary}
\begin{proof}
Since $\tilde S(t),t\in \R$ is a $C_0$-group of isometries, by the Weak Spectral Mapping Theorem (see e.g. \cite[Theorem 3.16, p. 283]{engnag}) the identity (\ref{spec}) is valid.
\end{proof}

In \cite{luoloiminmat} it is proved that a function $g\in BUC(\R^+,\X)$ is an asymptotic Bloch $1$-periodic function of type $p$ if and only if $\sigma (g)\subset \{ e^{ip}\}$. In the following another characterization is given in term of spectrum $sp(f)$ as in many circumstances it is easier to estimate this spectrum than $\sigma (f)$.
\begin{theorem}\label{the 3.8}
Let $g\in BUC  (\R^+,\X)$. Then,
\begin{enumerate}
\item
If $\xi_0$ is an isolated point in $sp (g)$, then $i\xi_0$ is either removable or a pole of ${\hat g(\lambda)}$ of order less than $1$;
\item If $sp  (g)=\emptyset$, then $g\in C_{0 }(\R^+,\X)$;
\item $sp  (g)$ is a closed subset of $\R$;
\item $g$ is anasymptotic Bloch $1$-periodic function of type $p$ if and only if $sp (g) \subset  p+2\pi \Z$.
\end{enumerate}
\end{theorem}
\begin{proof}
For the proofs of (i), (ii) and (iii) see \cite{luomin}. For (iv) note that as shown in \cite{luoloiminmat}, 
$f$ is an asymptotic Bloch $1$-periodic function of type $p$ if and only if $\sigma (f) \subset \{ e^{ip}\}$. By Collorary \ref{cor 1}, this is equivalent to $sp (f) \subset p+2\pi \Z$.
\end{proof}

%%%%%%%%%%%%%%%%%%%%%%%%%%%%%%%%%%%%%%%%%%%%%%%%%%%%%%%%%%%%%%%%%%%

\subsection{Applications to the asymptotic behavior of solutions of Eq(\ref{eq:0})}
We are going to apply the spectral theory of bounded functions in the previous section to study the asymptotic behavior of mild solutions of Eq.(\ref{fde}) with $0<\alpha \le 1$.

\begin{definition}
 A function $u\in BUC(\R^+,\X)$ is said to be an asymptotic mild solution of Eq.(\ref{fde}) if there exists a function $\epsilon (\cdot )\in C_0(\R^+,\X)$ such that $u$ is a mild solution of the equation
\begin{equation}\label{aeq}
D^{\alpha}_Cu(t)=Au(t)+f(t)+\epsilon(t), t\geq 0.
\end{equation}
\end{definition}
Below we will denote by
$\rho (A,\alpha)$
the set of all $\xi_0\in \C$ such that $(\xi_0^\alpha -A)$ has an inverse $(\xi_0^\alpha -A)^{-1}$ that is analytic in a neighborhood of $\xi_0$, and by $\Sigma_i (A,\alpha ):=i\R  \backslash \rho (A,\alpha)$. 

\medskip
We will use a new notation
\begin{equation}
R_\alpha (\lambda , A):= \lambda ^{\alpha-1} ( \lambda^{\alpha} -A)^{-1} .
\end{equation}
The lemma below is actually stated in \cite[Lemma 2.2]{arebat1}. Below is an adaptation made in \cite[Lemma 4.1]{luomin}
\begin{lemma}\label{lem 4.1}
Let $u\in BUC  (\R^+,\X)$, $\xi_0\in\R$, and let the function $G(\lambda )$ (in $\lambda$) be an analytic extension of the function $\hat u(\lambda) =(\lambda -\tilde{\cal D})^{-1}\tilde u$ with $\Re\lambda >0$ on the open disk $B( i\xi_0,r)$ with some positive $r$. Then, $G(\lambda ) =\hat u (\lambda) $ for $\Re \lambda <0$ on a disk $B(i\xi_0,r)$.
\end{lemma}

\begin{corollary}\label{cor 4.2}
Let $u\in BUC   (\R^+,\X)$ be an asymptotic mild solution of Eq.(\ref{fde}), where $f\in BUC(\R^+,\X)$. Then,
\begin{equation}
i \cdot sp   (u) \subset \Sigma_i (A,\alpha )\cup i\cdot sp(f) .
\end{equation}
\end{corollary}
\begin{proof}
See \cite[Lemma 4.2]{luomin}. 
\end{proof}
The main results of this paper are the following:
\begin{theorem}\label{the main}
Assume that $\Sigma_i(A,\alpha)\subset 2i\pi \Z$ ($\Sigma_i(A,\alpha)\subset (2\Z+1)i\pi$, respectively) and $sp(f) \subset 2\pi \Z$ ($sp(f) \subset (2\Z+1)\pi$, respectively).
Then, every asymptotic mild solution of Eq.(\ref{fde}) is an asymptotic $1$ -periodic solution (an asymptotic anti $1$ -periodic solution, respectively). 
\end{theorem}
\begin{proof}
By Corollary \ref{cor 4.2} for every asymptotic mild solution $u\in BUC(\R^+,\R)$ has the property that
$$
i \cdot sp(u) \subset \Sigma_i (A,\alpha )\cup i\cdot sp(f).
$$
Therefore, $sp(u) \subset 2\pi \Z$ if both $\Sigma_i (A,\alpha )\cup i\cdot sp(f)$ are parts of $2i\pi \Z$ , so by Theorem \ref{the 3.8}, $u$ is asymptotic $1$-periodic. The case of asymptotic anti $1$-periodicity is treated in the same manner.
\end{proof}

\begin{lemma}
For $\alpha \in [0,1)$ the following assertions are valid
\begin{enumerate}
\item The operator $J^\alpha $ maps $C_0(\R^+,\X) \to C_0(\R^+,\X)$, so it induces a linear bounded operator $\tilde J^\alpha$ from $\Y \to \Y$.
\item 
Let $u\in BUC(\R^+,\X)$ be an asymptotic mild solution of Eq.(\ref{fde}) and $A$ be bounded linear operator. Then,
\begin{equation}\label{3.13.1}
\tilde S\tilde u =\tilde{A} \tilde{J^\alpha} \tilde S\tilde u+\tilde S\tilde J^\alpha \tilde f.
\end{equation}
\end{enumerate}
\end{lemma}
\begin{proof}
Part (i): The proof is obvious.\\
Part (ii):
\begin{align*}
\left[Su\right] (t) & = u (t+1)\\
&=A \int^{t+1}_0 g_\alpha (t+1-\tau )u(\tau )d\tau    +\int^{t+1}_0 g_\alpha (t+1-\tau )[f(\tau )+\epsilon (\tau)] d\tau     \\
&= A\left[ \int^{t+1}_1 g_\alpha (t+1-\tau )u(\tau )d\tau + \int^{1}_0 g_\alpha (t+1-\tau )u(\tau )d\tau \right] \\
&\hspace{0.51cm} + \int^{t+1}_1 g_\alpha (t+1-\tau )f(\tau )d\tau + \int^{1}_0 g_\alpha (t+1-\tau )[f(\tau )+\epsilon(\tau)]d\tau  \\
&= A\left[  \left[ J^\alpha Su\right] (t) +\int^1_0 \frac{(t+1-\tau)^{\alpha -1}}{\Gamma (\alpha)}u(\tau )d\tau \right]\\
&\hspace{1cm} +  J^\alpha S[f+\epsilon]  (t) +\int^1_0 \frac{(t+1-\tau)^{\alpha -1}}{\Gamma (\alpha)}[f(\tau )+\epsilon (\tau)]d\tau  \\
\end{align*}
Since $A$ is bounded and 
\begin{align*}
&\lim_{t\to\infty} \int^1_0 \frac{(t+1-\tau)^{\alpha -1}}{\Gamma (\alpha)}[f(\tau )+\epsilon (\tau)]d\tau   =0,\\
&\lim_{t\to\infty}   J^\alpha S\epsilon (t) =0.
\end{align*}
(\ref{3.13.1}) follows.
\end{proof}

\begin{lemma}\label{lem.3.12}
Let $u\in BUC(\R^+,\X)$ and  $A$ be bounded linear operator. Assume further that $u$ is an asymptotic mild solution of Eq.(\ref{fde}) if and only if
\begin{equation}\label{3.13}
 \tilde u =\tilde{A} \tilde{J^\alpha}  \tilde u+ \tilde J^\alpha  \tilde f.
\end{equation}
\end{lemma}
\begin{proof}
It suffices to show that if $u\in BUC(\R^+,\X)$ satisfies (\ref{3.13}), then it must be an asymptotic mild solution of Eq.(\ref{fde}). In fact,  there are functions $\epsilon_i(\cdot )\in C_0(\R^+,\X)$ such that
\begin{align}
u(t)+\epsilon_1(t) &= AJ^\alpha [u+\epsilon_1](t) +\epsilon_2(t)+J^\alpha f(t) +\epsilon_3(t)
\end{align}
Therefore, as $A$ is a bounded linear operator, 
\begin{align}
u(t)+&= AJ^\alpha u(t) +J^\alpha f(t) +\epsilon_4(t),
\end{align}
where $\epsilon_4(t):=AJ^\alpha \epsilon_1(t)- \epsilon_1(t)+\epsilon_2(t)+\epsilon_3(t)$. as $A$ is a bounded linear operator, and $\tilde J^\alpha$ are bounded,
$\epsilon_4\in C_0(\R^+,\X)$.
\end{proof}

\begin{theorem}\label{the main2}
Assume that $A$ is a bounded operator, $e^{\Sigma _i(A,\alpha)}\backslash \{ 1\}$ is closed, and $f$ is asymptotic $1$-periodic. Then, if Eq.(\ref{fde}) has an asymptotic 1-perioidc solution if and only if it has a bounded uniformly continuous asymptotic solution on $\R^+$. 
\end{theorem}
\begin{proof}
Consider the operator
$$
L\tilde u :=  \tilde u -(\tilde{A} \tilde{J^\alpha}  \tilde u+ \tilde J^\alpha  \tilde f).
$$
It is clear that $u\in BUC(\R^+,\X)$ is an asymptotic mild solution of Eq.(\ref{fde}) if and only if $L\tilde u =0$.
Set $\Lambda := e^{\Sigma _i(A,\alpha)}\cup \{ 1\}$, $\Lambda_1 := e^{\Sigma_i (A,\alpha)}\backslash \{ 1\}$ and $\Lambda_2:= \{1\}$. Then, by the assumption $\Lambda_1$ and $\Lambda_2$ are two disjoint closed subset of the unit circle $\Gamma$. By Corollary \ref{cor 4.2}, $\sigma (u)\subset \Lambda $. 
Moreover, by \cite[Assertion (iii), Proposition 3.4]{luoloiminmat}
there exists a projection $P:\Y_\Lambda \to \Y_{\Lambda_2}$ that commutes with any bounded operator that commutes with $\tilde S$. therefore, 
as $\tilde SL = L\tilde S$, we have $PL=LP$. Consequently, 
\begin{align*}
0&= PL\tilde u \\
 &= LP\tilde u \\
 &= P\tilde u -(\tilde{A} \tilde{J^\alpha} P \tilde u+ \tilde J^\alpha  \tilde f).
\end{align*}
That means, $P\tilde u$ contains an asymptotic mild solution of Eq.(\ref{fde}). Since $P\tilde u\in \Y_{\Lambda_2}$ the spectrum of this solution must be part of $\Lambda_2=\{1\}$. By \cite[Proposition 3.4]{luoloiminmat}, this solution is asymptotic $1$-periodic.
\end{proof}
\begin{remark}
The above theorem is an analog version of the famous result in \cite{mas} for fractional differential equations. Another analog for asymptotic anti-$1$-periodic solutions can be obtained by using the set $\{-1\}$ instead of $\{1\}$ in the statements of the theorem.
\end{remark}
\subsection{The case with unbounded $A$}
Now we can state an analog of Theorem \ref{the main2} in the case of $A$ is an unbounded operator. 
We assume that $A: D(A)\subset X\to X$ is a closed linear operator which generates a $C_0$-semigroup $\{T(t)\}_{t\geq 0}$ in $X$ and $\{\lambda^\alpha:Re\lambda>0\}\subset \rho(A)$. Taking the Laplace transforms of both sides of \eqref{fde} gives
\begin{eqnarray}
	\hat u(\lambda ) &=& \hat S_\alpha (\lambda )x +\widehat{\big((\cdot )^{\alpha-1}  P_\alpha \big)}(\lambda )\cdot \hat f(\lambda ) .
\end{eqnarray}
where $\widehat{(\cdot )^{\alpha-1} P_\alpha}(\lambda)=(\lambda^\alpha-A)^{-1}, \hat S_\alpha(\lambda)=\lambda^{\alpha-1}(\lambda^\alpha-A)^{-1}$.
The inversion of the Laplace transform shows that $u(t)$ has the following form
$$
u(t)=S_\alpha(t)x+\int_0^t (t-s)^{\alpha-1}P_\alpha(t-s)f(s)ds.
$$
Hence, asymptotic mild solutions of  \eqref{fde} are defined as functions $u\in BUC(\R^{+},\X)$ satisfy  
\begin{equation}\label{eq:1a}
	u(t)=S_\alpha(t)x+\int_0^t (t-s)^{\alpha-1}P_\alpha(t-s)(f(s)+\epsilon(s))ds, \; t\geq 0.
\end{equation}
By the subordination principle (see \cite{baz}), $S_\alpha$ and $P_\alpha, \alpha \in (0,1]$, exist if $A$ generates a $C_0$- semigroup $\{T(t)\}_{t\geq 0}$. The explicit  formulas of  $S_\alpha$ and $P_\alpha$ were given in \cite{baz, Zhou}:
\begin{align*} 
	S_\alpha(t)&=\int_0^\infty\Phi_\alpha(s)T(t^\alpha s)ds,\\
	P_\alpha(t)&=\alpha\int_0^\infty s\Phi_\alpha(s)T(t^\alpha s)ds,
\end{align*}
where $\Phi_\alpha$ is a probability density function defined on $(0,\infty)$, that is, $\Phi_\alpha(t)\geq 0$ and
$\int_0^\infty\Phi_\alpha(t) dt=1$.  Assume that the semigroup $\{T(t)\}_{t\geq 0}$ generated by $A$ is exponentially stable, i.e., there are positive numbers $a, M$ such that
$$
\|T(t)\|\leq Me^{-at}, t\geq 0.
$$
By the fact that (see  \cite{Wang}) 
	$$\int_0^\infty \Phi_\alpha(\theta)e^{-z\theta}d\theta=E_{\alpha,1}(-z),$$
	$$ \int_0^\infty \alpha\theta\Phi_\alpha(\theta)e^{-z\theta}d\theta=E_{\alpha,\alpha}(-z),$$
	we have
	\begin{align*}
		\|S_\alpha(t)\|&\leq \int_0^\infty \Phi_\alpha(\theta)\|T(\theta t^\alpha)\|d\theta\\ 
		& \leq \int_0^\infty \Phi_\alpha(\theta)e^{-at^{\alpha}\theta}d\theta=E_{\alpha,1}(-at^{\alpha}),
	\end{align*}
  and 
 \begin{align*}
 	\|P_\alpha(t)\|& \leq \alpha \int_0^\infty\theta \Phi_\alpha(\theta)\|T(\theta t^\alpha)\|d\theta\\ 
 	& \leq M \alpha \int_0^\infty\theta \Phi_\alpha(\theta)e^{-a t^\alpha \theta}d\theta\\
 	&=ME_{\alpha,\alpha}(-a t^\alpha ).
 \end{align*}
 where $E_{\alpha,\beta}$, the Mittag-Leffler function, is defined as follows
$$E_{\alpha,\beta}(z)=\sum\limits_{n=0}^\infty\frac{z^n}{\Gamma(\alpha n+\beta)},\; \alpha,\beta>0,\; z\in \C.$$ 
 For $\alpha \in (0,1), t\in \R$, these Mittag-Leffler functions  have 
\begin{align*}
\lim\limits_{t\to -\infty}E_{\alpha,1}(t)=\lim\limits_{t\to -\infty}E_{\alpha,\alpha}(t)=0.
\end{align*}
From the estimates above, we obtain $$\|S_\alpha(t)\|, \|P_\alpha(t)\| \to 0, t\to\infty.$$
and 
\begin{align*}
	\int_{-\infty}^t (t-s)^{\alpha-1}\|P_\alpha(t-s)\| ds&\leq M \int_{-\infty}^t (t-s)^{\alpha-1}E_{\alpha,\alpha}(-a (t-s)^\alpha )ds\\
	&\leq \frac{M}{a}E_{\alpha,1}(-a(t-s)^\alpha)\bigg|_{-\infty}^t=\frac{M}{a}.
\end{align*}

It means that $S_\alpha(t)x\in C_0(\R^+,\X)$. By the same argument as in the proof Lemma \ref{lem.3.12}, we have the following result.
\begin{lemma}\label{lem 3.14}
Assume that $A$ generate a exponentially stable semigroup and $\{\lambda^\alpha: Re\lambda>0\} \subset \rho(A)$. Then, $u\in BUC(\R^+,\X)$  is an asymptotic mild solution of Eq.(\ref{fde}) if and only if
	\begin{equation}\label{3.13a}
		\tilde u = \tilde F_\alpha  \tilde f,
	\end{equation}
where $F_\alpha f (t)=\int\limits_0^t (t-s)^{\alpha-1}P_\alpha(t-s)f(s)ds, \; t>0.$
\end{lemma}
In the following we assume that $e^{\Sigma_i (A,\alpha)}\backslash \{ 1\}$ is closed, and $f$ is asymptotic $1$-periodic. Then, if we use the notations in the proof of Theorem \ref{the main2} with 
$\Lambda := e^{\Sigma_i (A,\alpha)}\cup \{ 1\}$, $\Lambda_1 := e^{\Sigma (A,\alpha)}\backslash \{ 1\}$ and $\Lambda_2:= \{1\}$ and the projection $P:\Y_\Lambda \to \Y_{\Lambda_2}$.
\begin{lemma}\label{lem 3.15}
Under the assumption of Lemma \ref{lem 3.14} we assume that $e^{\Sigma_i (A,\alpha)}\backslash \{ 1\}$ is closed, and $f$ is asymptotic $1$-periodic. Then, if $u\in BUC(\R^+,\X)$ is an asymptotic mild solution of Eq.(\ref{fde}), 
\begin{equation}
P\tilde u = \tilde F_\alpha \tilde f.
\end{equation}
\end{lemma}
\begin{proof}
 The operator $F_\alpha$ maps $C_0(\R^+,\X) \to C_0(\R^+,\X)$, so it induces a linear bounded operator $\tilde F_\alpha$ from $\Y \to \Y$.
If $g\in C_0(\R^+,\X)$, then, for each $\epsilon>0$, there exists a number $T>0$ such that $\|g\|\leq \epsilon$ for all $s\geq T.$
Then for all $t\geq 2T$, we have
\begin{align*}
\|F_\alpha f(t)\|\leq &\|g\|\int_0^{t/2}(t-s)^{\alpha -1}\|P_\alpha(t-s)\|ds+\epsilon\int_{t/2}^t(t-s)^{\alpha-1}\|P_\alpha(t-s)\|ds\\
 \leq & \|g\|\int_{0}^{t/2} (t-s)^{\alpha-1}E_{\alpha,\alpha}(-a (t-s)^\alpha )ds +\epsilon\frac{M}{a}\\
 =&\|g\|(E_{\alpha,1}(-at^\alpha)-E_{\alpha,1}(-a(t/2)^\alpha)) +\epsilon\frac{M}{a}.
\end{align*}
 
 Moreover,  
 \begin{align*}
  SF_\alpha f(t)= & \int_{0}^{t+1}(t+1-s)^{\alpha-1}P_\alpha(t+1-s)f(s)ds\\
 = & \int_{1}^{t+1}(t+1-s)^{\alpha-1}P_\alpha(t+1-s)f(s)ds+\int_{0}^{1}(t+1-s)^{\alpha-1}P_\alpha(t+1-s)f(s)ds\\
 	=& F_\alpha Sf(t) +h(t).
 \end{align*}
 Note that 
 $$\|h(t)\|=\|\int_{0}^{1}(t+1-s)^{\alpha-1}P_\alpha(t+1-s)f(s)ds\|\leq \|f\|(E_{\alpha,1}(-a(t+1)^\alpha)-E_{\alpha,1}(-a(t)^\alpha)),$$
 it means that $h\in C_0(\R^+,\X)$ and 
 $\tilde F_\alpha $ commutes with $\tilde S$. 
 Using notations in Theorem \ref{the main2}, it easy to see that if $u\in BUC(\R^+,\X)$ is asymptotic mild solution of \eqref{fde}, then
 \begin{align*}
  0=&P(\tilde u- \tilde F_\alpha  \tilde f)\\
  =&P\tilde u- P\tilde F_\alpha \tilde f\\
  =&P\tilde u- \tilde F_\alpha \tilde f. 
  \end{align*}
  \end{proof}
As a consequence of Lemma \ref{lem 3.15} we note that $\sigma(F_\alpha f)\subset \sigma(f).$
From arguments above, we also obtain the following theorem.
 \begin{theorem}\label{the main3}
 Assume that $A$ generate a exponentially stable semigroup and $\{\lambda^\alpha: Re\lambda>0\} \subset \rho(A)$, $e^{\Sigma _i(A,\alpha)}\backslash \{ 1\}$ is closed, and $f$ is asymptotic $1$-periodic. Then, if Eq.(\ref{fde}) has an asymptotic 1-periodic solution if and only if it has a bounded uniformly continuous asymptotic solution on $\R^+$. 
 \end{theorem}
\subsection{Examples}

%%%%%%%%%%%%%%%%%%%%%%%%%%%%%%%%%%%%%%%%%%%
\begin{example}\label{3.18}
	Consider the initial value problem 
	\begin{equation}\label{exa1}
		\begin{cases}
			D^\alpha_t u(x,t)= a u_{xx}(x,t) + f(x,t), \ \ x\in \Omega=(0,\pi),\ t> 0, \cr
			u(0,t)=u(\pi ,t)=0 , \ \ t > 0,
			\cr
			u(x,0)=u_0(x),\; x\in \Omega,
		\end{cases}
	\end{equation}
	where $D^\alpha_t u(x,t)$ is the fractional derivative in Caputo's sense in $t$ of degree $\alpha\in (0,1)$, $u(x,t), f(x,t)$ are scalar-valued functions such that $f(\cdot , t)\in L^2(\Omega)$ for each $t\in\R^+$, $f:\R^+\ni t \mapsto f(\cdot ,t)\in L^2(\Omega)$ is uniformly continuous and bounded, and
	and $a>0$ is given. We define the space 
	${\X}:=L^2(\Omega)$ and $A: {\X}\to {\X}$ by the formula
	\begin{equation}\label{exa 3}
		\begin{cases}\hspace{.5cm}
			Ay=ay'', \cr
			D(A)=H^1_0(\Omega)\cap H^2(\Omega).
		\end{cases}
	\end{equation}
	Problem \ref{exa1} can be rewritten  in  the form of an evolution equation 
	\begin{equation}\label{4.3}
		D^\alpha_t u(t)=Au(t) + f(t), \   t\in \R^+, u(0)=u_0,
	\end{equation}
	where $u(t) \in {\X}$, $A$ is defined as above. 
	As is well known the spectrum of $A$ consists of eigenvalues that satisfy this equation
	$$
	\left\{\begin{array}{l}
		\dfrac{d^{2} u}{d x^{2}}=\frac{1}{a}\lambda u \\
		u(0)=u(\pi)=0
	\end{array} \quad \text { in } \Omega . \right .
	$$
	This implies
	$$
	\lambda_{n}=- an^{2} , u_{n}=C \sin (n x), n \in \N
	$$
	Furthermore, $A$ generates a compact semigroup $S(t)$ with
	$$
	\|S(t)\| \leq M \cdot e^{- at}, \forall t \geq 0
	$$
	According to the subordination principle, $A$ generates the subordinated resolvent $S_{\alpha}(t), \alpha \in(0 ; 1)$ such that $\lim _{t \rightarrow \infty}\left\|S_{\alpha}(t)\right\|=0.$ By definition,
	$$
	\Sigma_i(A, \alpha)=\{ \lambda \in i\R : \ \left(\lambda^{\alpha}-A\right)^{-1} \ \mbox{does not exists as a bounded operator}\}.
	$$
	As
	$$
	\lambda^\alpha  \in \sigma (A) \Longleftrightarrow \lambda^{\alpha}=-an^{2}, n=1,2, . .
	$$
	we have
	$$
	\begin{gathered}
		\lambda^{\alpha}=|\lambda|^{\alpha} \cdot e^{i \alpha \arg \lambda} \\
		\lambda^{\alpha}=-an^{2} = an^{2} e^{i \pi}, n=1,2, . .
	\end{gathered}
	$$
	this is equivalent to
	$$
	\left\{\begin{array} { l } 
		{ | \lambda | ^ { \alpha } = an ^ { 2 } , n = 1 , 2 , . . } \\
		{ \operatorname { a r g } \lambda = \frac { \pi } { \alpha } }
	\end{array} \Rightarrow \left\{\begin{array}{l}
		|\lambda|=a^{\frac{1}{\alpha}}n^{\frac{2}{\alpha}}, n=1,2, . . \\
		\arg \lambda=\frac{\pi}{\alpha}
	\end{array}\right.\right.
	$$
	If the choices are $\alpha=\frac{2}{3}, a= \sqrt[3]{\pi^2}$, we get:
	$$
	\lambda_{n}=n^3 \pi e^{i \cdot \frac{3 \pi}{2}}=-in^3\pi , n=1,2, . .
	$$
	Therefore
	$$
	\Sigma_i(A, 2/3)=\left\{-in^3\pi, n=1,2, . .\right\}
	$$
	and 
	$$e^{\Sigma_i (A,2/3)}=  \{ 1\}\cup \{-1\}.$$
		
If $f(\cdot ,t)$ is asymptotic $1$-periodic in $t$, then $e^{i\cdot sp(f)} =\{1\}$. Therefore, by Theorem \ref{the main3}, there is an asymptotic 1-periodic solution $u$ of \eqref{4.3} if and only if there exists an asymptotic mild solution to Eq.(\ref{4.3}). Similar conclusions can be made if $f$ is asymptotic $1$-anti-periodic. This is an analog of Massera Theorem for the fractional differential equation (\ref{4.3}).
\end{example}

 \begin{example}
 We consider Problem \ref{exa1} again with different values of $\alpha$ to illustrate a Katznelson-Tzafriri type result. We will assume that $\alpha$ is a number that satisfies
 $$\frac{\pi}{\alpha}=\frac{\pi}{2}+k\pi,$$
 for a certain $ k=1,2,...$ For instance,
 $$
	\alpha=\frac{2}{2k+1}, 
$$
where $k$ is a positive integer. Then 
$$
\lambda_{n}=\left((2\pi)^{\frac{2}{2k+1}}\cdot n^{2}\right)^{\frac{2k+1}{2}}\cdot e^{i\cdot {\frac{2k+1}{2}\pi}}$$
$$
 \lambda_{n}=\pm i2n^{(2k+1)}\pi=2mi\pi , 
$$
where $ m\in \N^* $.

\medskip
Therefore, assuming that $f(\cdot )$ is asymptotic $1$-periodic, by Theorem \ref{the main}, if $u\in BCU(\R^+,\X)$ is an asymptotic mild solution of Eq. (\ref{4.3}), $u$ must be asymptotic $1$-periodic.
 
 \end{example}
 
\bibliographystyle{amsplain}

\end{document}